\documentclass[a4paper,11pt,reqno]{amsart}
\usepackage{a4wide}
\usepackage{amsfonts,amsmath,amssymb}
\usepackage{graphicx}
\usepackage{tikz}

\newtheorem{theo}{Theorem}
\newtheorem{prop}[theo]{Proposition}
\newtheorem{lemma}[theo]{Lemma}

\theoremstyle{remark}
\newtheorem{rem}{Remark}

\newcommand{\ns}{\operatorname{ns}}
\newcommand{\nr}{\operatorname{nr}}

\begin{document}

\thispagestyle{plain}

\title{On the number of nonisomorphic subtrees of a tree}

\author{\'Eva Czabarka, L\'aszl\'o A. Sz\'ekely }
\address{Department of Mathematics \\ University of South Carolina \\ Columbia, SC 29208 \\ USA}
\email{\{czabarka,szekely\}@math.sc.edu }
\thanks{The second author was supported in part by the  NSF DMS,  grant number 1300547,  the third author was supported in part by the National
Research Foundation of South Africa, grant number 96236.}
\author{Stephan Wagner}
\address{Department of Mathematical Sciences \\ Stellenbosch University \\ Private Bag X1, Matieland 7602 \\ South Africa}
\email{swagner@sun.ac.za}

\date{\today}

\begin{abstract}
We show that a tree of order $n$ has at most $O(5^{n/4})$ nonisomorphic subtrees, and that this bound is best possible. We also prove an analogous result for the number of nonisomorphic rooted subtrees of a rooted tree.
\end{abstract}

\maketitle

\section{introduction}

Subtrees of a tree have been studied extensively: Jamison \cite{jamison1983average, jamison1984monotonicity} investigated the average number of vertices
in a subtree, Sz\'ekely and Wang studied the number of subtrees of trees \cite{subtrees2005, largest2006}.
Chung, Graham and Coppersmith \cite{universal} found the smallest order (asymptotically) of a tree that contains all $n$-vertex  trees as subtrees.
A number of extremal results are known: in particular, it is known that a tree of order $n$ has at least $\binom{n+1}{2}$ subtrees (with equality for the path) and at most $2^{n-1}+1$ subtrees (with equality for the star).

 Things change considerably, however, if one considers the number of distinct \emph{nonisomorphic} subtrees.
Bubeck and Linial \cite{BubeckLinial} recently analyzed the distribution of subtrees of fixed order by isomorphism type. In this paper we consider the extremal problem of determining the smallest and largest number of nonisomorphic subtrees of a tree.
Both the path and the star have only very few nonisomorphic subtrees (equal to the number of vertices, to be precise), and this is in fact the minimum:

\begin{prop}
Every tree of order $n$ has at least $n$ distinct nonisomorphic subtrees.
\end{prop}

\begin{proof}
This is easily established by noticing that every tree of order $n$ has subtrees of every order $k$ between $1$ and $n$ (obtained by repeatedly removing leaves from the original tree), which are trivially nonisomorphic.
\end{proof}

The maximum is much more difficult to obtain; we will show that it is of the order $\Theta(5^{n/4})$. In a certain sense, this is a dual question to the aforementioned problem of Chung, Graham and Coppersmith: while they were looking for the minimum number of vertices needed to contain all small trees, we would like to know how many different trees can fit into a tree with a given number of vertices.

For our purposes, it turns out to be useful to also consider a closely related problem: for a rooted tree, we count the number of nonisomorphic subtrees that contain the root, where non-isomorphism is understood in the rooted sense, i.e., two subtrees containing the root are considered isomorphic if there is an isomorphism between them that maps the root to itself. Again, the maximum number is $\Theta(5^{n/4})$, as we will prove in the following. Let us first introduce some notation. 

For a tree $T$, we let $|T|$ be the number of vertices and $\ns(T)$ be the number of nonisomorphic subtrees of $T$. Likewise, for a rooted tree $T$, we let $\nr(T)$ be the number of nonisomorphic subtrees containing the root, in the sense explained in the previous paragraph. Moreover, we write
$$S_n = \max_{|T| = n} \ns(T) \qquad \text{and} \qquad R_n = \max_{|T| = n} \nr(T)$$
for the respective maxima among trees with $n$ vertices. The following table gives explicit values for small $n$, obtained by means of a comprehensive computer search:

\begin{table}[h]
\begin{center}
\begin{tabular}{|l||c|c|c|c|c|c|c|c|c|c|}
\hline
$n$ &1 &2 &3 &4 &5 &6 &7 &8 &9 &10 \\
\hline
$S_n$ &1 &2 &3 &4 &6 &8 &11 &16 &23 &33 \\
\hline
$R_n$ &1 &2 &3 &5 &7 &11 &16 &24 &34 &54 \\
\hline
\end{tabular}
\medskip
\end{center}
\caption{$S_n$ and $R_n$ for small values of $n$.}
\end{table}

While it appears difficult to determine $S_n$ and $R_n$ explicitly for general $n$, we will be able to bound them from above and to provide a construction that matches the bound up to a constant factor.

\section{The asymptotic order of $S_n$ and $R_n$}

Our approach will consist of the following steps:
\begin{itemize}
\item Provide a construction that yields trees with ``many'' nonisomorphic subtrees,
\item Bound $S_n$ in terms of $R_n$,
\item Prove an upper bound on $R_n$.
\end{itemize}

The three statements of Proposition~\ref{prop:main} below correspond to those three steps. In order to prove the upper bound on $R_n$, we will require the following simple lemma.

\begin{lemma}\
\begin{enumerate}
\item[1.] Let $T$ be a rooted tree, and suppose that $T$ is the union of two rooted trees $R_1$ and $R_2$ (nontrivial, i.e. of order at least $2$) that only share the root. Then we have the inequality
\begin{equation}\label{eq:upper_bound1}
\nr(T) \leq \nr(R_1) \nr(R_2) - 1.
\end{equation}
\item[2.] Let $T_1,T_2,\ldots,T_d$ be the root branches (in other words, the components that remain when the root is removed) of a rooted tree $T$, endowed with their natural roots. Then the inequality
\begin{equation}\label{eq:upper_bound2}
\nr(T) \leq \prod_{j=1}^d (\nr(T_j) + 1)
\end{equation}
holds.
\end{enumerate}
\end{lemma}
\begin{proof}\
\begin{enumerate}
\item[1.] Note that each subtree $S$ of $T$ that contains the subtree decomposes naturally into a subtree $S_1$ of $R_1$ and a subtree of $S_2$ of $R_2$, each containing the respective root. Two pairs of subtrees $(S_1,S_2)$ and $(S_1',S_2')$ of subtrees such that $S_1$ is isomorphic to $S_1'$ and $S_2$ is isomorphic to $S_2'$ induce isomorphic subtrees of $T$. Therefore, we clearly have
$$\nr(T) \leq \nr(R_1) \nr(R_2).$$
To show that even strict inequality must hold for nontrivial trees $R_1$ and $R_2$ (which implies the desired statement), note that the two-vertex subtree is counted twice in this argument: once as a subtree of $R_1$, once as a subtree of $R_2$ (of course, this might also apply to other subtrees). Thus we obtain~\eqref{eq:upper_bound1}.
\item[2.] The argument is analogous to the first part: every subtree of $T$ induces a subtree that contains the root or the empty set in each $T_j$. The inequality~\eqref{eq:upper_bound2} follows immediately (and it is generally strict because subtrees are counted repeatedly).
\end{enumerate}
\end{proof}

\begin{prop}\label{prop:main}
The following inequalities hold for all $n \geq 1$:
\begin{enumerate}
\item[1.] $S_n \geq 2 \cdot 5^{n/4-2}$,
\item[2.] $S_n \leq R_n + 3 \cdot 2^{n/2-1}$,
\item[3.] $R_n \leq 5^{n/4}$.
\end{enumerate}
\end{prop}

\begin{proof}\ 

\begin{enumerate}
\item[1.] For $n < 8$, the stated inequality is essentially trivial, since it provides a lower bound of $1$ ($n < 7$) or $2$ ($n=7$). For $n \geq 8$, we obtain the lower bound on $S_n$ by an explicit construction (see Figure~\ref{fig:construction}). Its core is formed by a path of $m = \lfloor \frac{n}{4} \rfloor -2$ vertices, to each of which we attach three vertices: one leaf and two others forming a path of length $2$. Depending on the residue class of $n$ modulo $4$, we add between $8$ and $11$ additional vertices as indicated in the figure to obtain an $n$-vertex tree that we denote by $C_n$.

\begin{figure}[htbp]
\begin{center}
\begin{tikzpicture}

	\node at (-2,1.2) {$n \equiv 3 \bmod 4$:};

        \node[fill=black,circle,inner sep=2pt]  at (-3,0) {};
        \node[fill=black,circle,inner sep=2pt]  at (-2,0) {};
        \node[fill=black,circle,inner sep=2pt]  at (-1,0) {};

        \node[fill=black,circle,inner sep=2pt]  at (0,0) {};
        \node[fill=black,circle,inner sep=2pt]  at (1,0) {};
        \node[fill=black,circle,inner sep=2pt]  at (2,0) {};
        \node[fill=black,circle,inner sep=2pt]  at (6,0) {};
        \node[fill=black,circle,inner sep=2pt]  at (7,0) {};

        \node[fill=black,circle,inner sep=2pt]  at (8,0) {};
        \node[fill=black,circle,inner sep=2pt]  at (9,0) {};
        \node[fill=black,circle,inner sep=2pt]  at (10,0) {};

        \node[fill=black,circle,inner sep=2pt]  at (-2,.5) {};
        \node[fill=black,circle,inner sep=2pt]  at (-.7,.4) {};
        \node[fill=black,circle,inner sep=2pt]  at (-1.3,.4) {};
        \node[fill=black,circle,inner sep=2pt]  at (7.7,.4) {};
        \node[fill=black,circle,inner sep=2pt]  at (8.3,.4) {};

	\node at (0,0.3) {$v_1$};
	\node at (1,0.3) {$v_2$};
	\node at (2,0.3) {$v_3$};
	\node at (6,0.3) {$v_{m-1}$};
	\node at (7,0.3) {$v_m$};

	\node at (-3,-0.4) {$x_1$};
	\node at (-2,-0.4) {$x_2$};
	\node at (-1,-0.4) {$x_3$};

	\node at (8,-0.4) {$y_1$};
	\node at (9,-0.4) {$y_2$};
	\node at (10,-0.4) {$y_3$};

	\node at (-2.3,0.5) {$z$};

	\draw (-3,0)--(2,0);
	\draw [dotted] (2,0)--(6,0);
	\draw (6,0)--(10,0);

	\draw (-2,0)--(-2,.5);
	\draw (-1.3,.4)--(-1,0)--(-.7,.4);
	\draw (8.3,.4)--(8,0)--(7.7,.4);

        \node[fill=black,circle,inner sep=2pt]  at (-0.3,-0.4) {};
        \node[fill=black,circle,inner sep=2pt]  at (0.3,-0.4) {};
        \node[fill=black,circle,inner sep=2pt]  at (-0.3,-0.9) {};
	\draw (-0.3,-0.9)--(-0.3,-0.4)--(0,0)--(0.3,-0.4);

        \node[fill=black,circle,inner sep=2pt]  at (0.7,-0.4) {};
        \node[fill=black,circle,inner sep=2pt]  at (1.3,-0.4) {};
        \node[fill=black,circle,inner sep=2pt]  at (0.7,-0.9) {};
	\draw (0.7,-0.9)--(0.7,-0.4)--(1,0)--(1.3,-0.4);

        \node[fill=black,circle,inner sep=2pt]  at (1.7,-0.4) {};
        \node[fill=black,circle,inner sep=2pt]  at (2.3,-0.4) {};
        \node[fill=black,circle,inner sep=2pt]  at (1.7,-0.9) {};
	\draw (1.7,-0.9)--(1.7,-0.4)--(2,0)--(2.3,-0.4);

        \node[fill=black,circle,inner sep=2pt]  at (5.7,-0.4) {};
        \node[fill=black,circle,inner sep=2pt]  at (6.3,-0.4) {};
        \node[fill=black,circle,inner sep=2pt]  at (5.7,-0.9) {};
	\draw (5.7,-0.9)--(5.7,-0.4)--(6,0)--(6.3,-0.4);

        \node[fill=black,circle,inner sep=2pt]  at (6.7,-0.4) {};
        \node[fill=black,circle,inner sep=2pt]  at (7.3,-0.4) {};
        \node[fill=black,circle,inner sep=2pt]  at (6.7,-0.9) {};
	\draw (6.7,-0.9)--(6.7,-0.4)--(7,0)--(7.3,-0.4);

	\node at (-2,4.2) {$n \equiv 2 \bmod 4$:};

        \node[fill=black,circle,inner sep=2pt]  at (-3,3) {};
        \node[fill=black,circle,inner sep=2pt]  at (-2,3) {};
        \node[fill=black,circle,inner sep=2pt]  at (-1,3) {};

        \node[fill=black,circle,inner sep=2pt]  at (0,3) {};
        \node[fill=black,circle,inner sep=2pt]  at (1,3) {};
        \node[fill=black,circle,inner sep=2pt]  at (2,3) {};
        \node[fill=black,circle,inner sep=2pt]  at (6,3) {};
        \node[fill=black,circle,inner sep=2pt]  at (7,3) {};

        \node[fill=black,circle,inner sep=2pt]  at (8,3) {};
        \node[fill=black,circle,inner sep=2pt]  at (9,3) {};
        \node[fill=black,circle,inner sep=2pt]  at (10,3) {};

        \node[fill=black,circle,inner sep=2pt]  at (-2,3.5) {};
        \node[fill=black,circle,inner sep=2pt]  at (-1.3,3.4) {};
        \node[fill=black,circle,inner sep=2pt]  at (-.7,3.4) {};
        \node[fill=black,circle,inner sep=2pt]  at (8,3.5) {};

	\node at (0,3.3) {$v_1$};
	\node at (1,3.3) {$v_2$};
	\node at (2,3.3) {$v_3$};
	\node at (6,3.3) {$v_{m-1}$};
	\node at (7,3.3) {$v_m$};

	\node at (-3,2.6) {$x_1$};
	\node at (-2,2.6) {$x_2$};
	\node at (-1,2.6) {$x_3$};

	\node at (8,2.6) {$y_1$};
	\node at (9,2.6) {$y_2$};
	\node at (10,2.6) {$y_3$};

	\node at (-2.3,3.5) {$z$};

	\draw (-3,3)--(2,3);
	\draw [dotted] (2,3)--(6,3);
	\draw (6,3)--(10,3);

	\draw (-2,3)--(-2,3.5);
	\draw (-1.3,3.4)--(-1,3)--(-.7,3.4);
	\draw (8,3)--(8,3.5);

        \node[fill=black,circle,inner sep=2pt]  at (-0.3,2.6) {};
        \node[fill=black,circle,inner sep=2pt]  at (0.3,2.6) {};
        \node[fill=black,circle,inner sep=2pt]  at (-0.3,2.1) {};
	\draw (-0.3,2.1)--(-0.3,2.6)--(0,3)--(0.3,2.6);

        \node[fill=black,circle,inner sep=2pt]  at (0.7,2.6) {};
        \node[fill=black,circle,inner sep=2pt]  at (1.3,2.6) {};
        \node[fill=black,circle,inner sep=2pt]  at (0.7,2.1) {};
	\draw (0.7,2.1)--(0.7,2.6)--(1,3)--(1.3,2.6);

        \node[fill=black,circle,inner sep=2pt]  at (1.7,2.6) {};
        \node[fill=black,circle,inner sep=2pt]  at (2.3,2.6) {};
        \node[fill=black,circle,inner sep=2pt]  at (1.7,2.1) {};
	\draw (1.7,2.1)--(1.7,2.6)--(2,3)--(2.3,2.6);

        \node[fill=black,circle,inner sep=2pt]  at (5.7,2.6) {};
        \node[fill=black,circle,inner sep=2pt]  at (6.3,2.6) {};
        \node[fill=black,circle,inner sep=2pt]  at (5.7,2.1) {};
	\draw (5.7,2.1)--(5.7,2.6)--(6,3)--(6.3,2.6);

        \node[fill=black,circle,inner sep=2pt]  at (6.7,2.6) {};
        \node[fill=black,circle,inner sep=2pt]  at (7.3,2.6) {};
        \node[fill=black,circle,inner sep=2pt]  at (6.7,2.1) {};
	\draw (6.7,2.1)--(6.7,2.6)--(7,3)--(7.3,2.6);

	\node at (-2,7.2) {$n \equiv 1 \bmod 4$:};

        \node[fill=black,circle,inner sep=2pt]  at (-3,6) {};
        \node[fill=black,circle,inner sep=2pt]  at (-2,6) {};
        \node[fill=black,circle,inner sep=2pt]  at (-1,6) {};

        \node[fill=black,circle,inner sep=2pt]  at (0,6) {};
        \node[fill=black,circle,inner sep=2pt]  at (1,6) {};
        \node[fill=black,circle,inner sep=2pt]  at (2,6) {};
        \node[fill=black,circle,inner sep=2pt]  at (6,6) {};
        \node[fill=black,circle,inner sep=2pt]  at (7,6) {};

        \node[fill=black,circle,inner sep=2pt]  at (8,6) {};
        \node[fill=black,circle,inner sep=2pt]  at (9,6) {};
        \node[fill=black,circle,inner sep=2pt]  at (10,6) {};

        \node[fill=black,circle,inner sep=2pt]  at (-2,6.5) {};
        \node[fill=black,circle,inner sep=2pt]  at (-1,6.5) {};
        \node[fill=black,circle,inner sep=2pt]  at (8,6.5) {};

	\node at (0,6.3) {$v_1$};
	\node at (1,6.3) {$v_2$};
	\node at (2,6.3) {$v_3$};
	\node at (6,6.3) {$v_{m-1}$};
	\node at (7,6.3) {$v_m$};

	\node at (-3,5.6) {$x_1$};
	\node at (-2,5.6) {$x_2$};
	\node at (-1,5.6) {$x_3$};

	\node at (8,5.6) {$y_1$};
	\node at (9,5.6) {$y_2$};
	\node at (10,5.6) {$y_3$};

	\draw (-3,6)--(2,6);
	\draw [dotted] (2,6)--(6,6);
	\draw (6,6)--(10,6);

	\draw (-2,6)--(-2,6.5);
	\draw (-1,6)--(-1,6.5);
	\draw (8,6)--(8,6.5);

	\node at (-2.3,6.5) {$z$};

        \node[fill=black,circle,inner sep=2pt]  at (-0.3,5.6) {};
        \node[fill=black,circle,inner sep=2pt]  at (0.3,5.6) {};
        \node[fill=black,circle,inner sep=2pt]  at (-0.3,5.1) {};
	\draw (-0.3,5.1)--(-0.3,5.6)--(0,6)--(0.3,5.6);

        \node[fill=black,circle,inner sep=2pt]  at (0.7,5.6) {};
        \node[fill=black,circle,inner sep=2pt]  at (1.3,5.6) {};
        \node[fill=black,circle,inner sep=2pt]  at (0.7,5.1) {};
	\draw (0.7,5.1)--(0.7,5.6)--(1,6)--(1.3,5.6);

        \node[fill=black,circle,inner sep=2pt]  at (1.7,5.6) {};
        \node[fill=black,circle,inner sep=2pt]  at (2.3,5.6) {};
        \node[fill=black,circle,inner sep=2pt]  at (1.7,5.1) {};
	\draw (1.7,5.1)--(1.7,5.6)--(2,6)--(2.3,5.6);

        \node[fill=black,circle,inner sep=2pt]  at (5.7,5.6) {};
        \node[fill=black,circle,inner sep=2pt]  at (6.3,5.6) {};
        \node[fill=black,circle,inner sep=2pt]  at (5.7,5.1) {};
	\draw (5.7,5.1)--(5.7,5.6)--(6,6)--(6.3,5.6);

        \node[fill=black,circle,inner sep=2pt]  at (6.7,5.6) {};
        \node[fill=black,circle,inner sep=2pt]  at (7.3,5.6) {};
        \node[fill=black,circle,inner sep=2pt]  at (6.7,5.1) {};
	\draw (6.7,5.1)--(6.7,5.6)--(7,6)--(7.3,5.6);

	\node at (-2,10.2) {$n \equiv 0 \bmod 4$:};

        \node[fill=black,circle,inner sep=2pt]  at (-3,9) {};
        \node[fill=black,circle,inner sep=2pt]  at (-2,9) {};
        \node[fill=black,circle,inner sep=2pt]  at (-1,9) {};

        \node[fill=black,circle,inner sep=2pt]  at (0,9) {};
        \node[fill=black,circle,inner sep=2pt]  at (1,9) {};
        \node[fill=black,circle,inner sep=2pt]  at (2,9) {};
        \node[fill=black,circle,inner sep=2pt]  at (6,9) {};
        \node[fill=black,circle,inner sep=2pt]  at (7,9) {};

        \node[fill=black,circle,inner sep=2pt]  at (8,9) {};
        \node[fill=black,circle,inner sep=2pt]  at (9,9) {};
        \node[fill=black,circle,inner sep=2pt]  at (10,9) {};

        \node[fill=black,circle,inner sep=2pt]  at (-2,9.5) {};
        \node[fill=black,circle,inner sep=2pt]  at (-1,9.5) {};

	\node at (0,9.3) {$v_1$};
	\node at (1,9.3) {$v_2$};
	\node at (2,9.3) {$v_3$};
	\node at (6,9.3) {$v_{m-1}$};
	\node at (7,9.3) {$v_m$};

	\node at (-3,8.6) {$x_1$};
	\node at (-2,8.6) {$x_2$};
	\node at (-1,8.6) {$x_3$};

	\node at (8,8.6) {$y_1$};
	\node at (9,8.6) {$y_2$};
	\node at (10,8.6) {$y_3$};

	\node at (-2.3,9.5) {$z$};

	\draw (-3,9)--(2,9);
	\draw [dotted] (2,9)--(6,9);
	\draw (6,9)--(10,9);

	\draw (-2,9)--(-2,9.5);
	\draw (-1,9)--(-1,9.5);

        \node[fill=black,circle,inner sep=2pt]  at (-0.3,8.6) {};
        \node[fill=black,circle,inner sep=2pt]  at (0.3,8.6) {};
        \node[fill=black,circle,inner sep=2pt]  at (-0.3,8.1) {};
	\draw (-0.3,8.1)--(-0.3,8.6)--(0,9)--(0.3,8.6);

        \node[fill=black,circle,inner sep=2pt]  at (0.7,8.6) {};
        \node[fill=black,circle,inner sep=2pt]  at (1.3,8.6) {};
        \node[fill=black,circle,inner sep=2pt]  at (0.7,8.1) {};
	\draw (0.7,8.1)--(0.7,8.6)--(1,9)--(1.3,8.6);

        \node[fill=black,circle,inner sep=2pt]  at (1.7,8.6) {};
        \node[fill=black,circle,inner sep=2pt]  at (2.3,8.6) {};
        \node[fill=black,circle,inner sep=2pt]  at (1.7,8.1) {};
	\draw (1.7,8.1)--(1.7,8.6)--(2,9)--(2.3,8.6);

        \node[fill=black,circle,inner sep=2pt]  at (5.7,8.6) {};
        \node[fill=black,circle,inner sep=2pt]  at (6.3,8.6) {};
        \node[fill=black,circle,inner sep=2pt]  at (5.7,8.1) {};
	\draw (5.7,8.1)--(5.7,8.6)--(6,9)--(6.3,8.6);

        \node[fill=black,circle,inner sep=2pt]  at (6.7,8.6) {};
        \node[fill=black,circle,inner sep=2pt]  at (7.3,8.6) {};
        \node[fill=black,circle,inner sep=2pt]  at (6.7,8.1) {};
	\draw (6.7,8.1)--(6.7,8.6)--(7,9)--(7.3,8.6);
\end{tikzpicture}
\end{center}
\caption{A construction that yields a lower bound.}\label{fig:construction}
\end{figure}
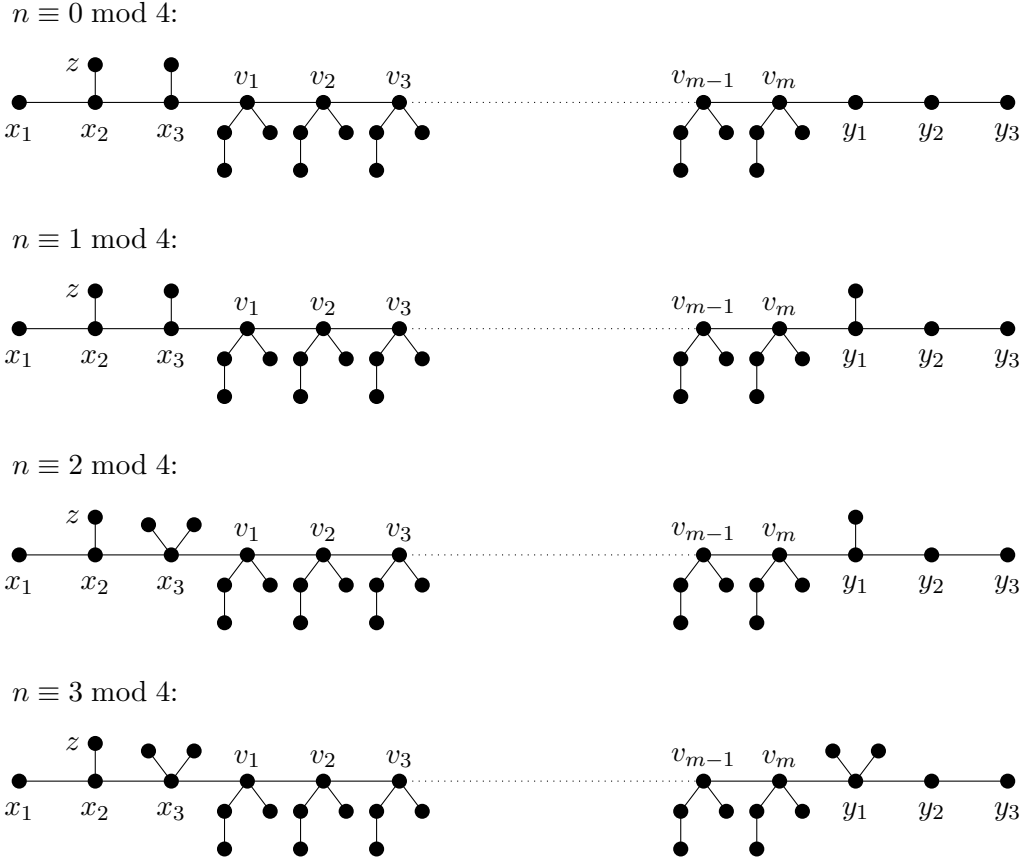

It is clear that $S_n \geq \ns(C_n)$, so let us estimate $\ns(C_n)$. For a simple lower bound, we only consider subtrees that contain the entire ``backbone'' consisting of $v_1,v_2,\ldots,v_m$ and $x_1,x_2,x_3,y_1,y_2,y_3$, as well as the vertex $z$. Note that the paths from $x_1$ to $y_3$ and from $z$ to $y_3$ are always the only diameters of these subtrees, so they are uniquely determined by the parts ``dangling'' from $x_3,v_1,v_2,\ldots,v_m,y_1$. Including $z$ in all subtrees we are counting guarantees that there is no double-counting due to mirror symmetry.

The leaf or leaves attached to $x_3$ can be included in such a subtree or not, which gives us $2$ possibilities for $n \equiv 0,1 \bmod 4$ and $3$ possibilities for $n \equiv 2,3 \bmod 4$. Likewise, the leaf or leaves attached to $y_1$ can be included or not, giving us $1,2$ or $3$ possible options, depending on the residue class again. Finally, for each of the vertices $v_1,v_2,\ldots,v_m$, we have $5$ distinct ways of extending the subtree by adding a subset of the three vertices attached to it. Altogether, this gives us
$$\ns(C_n) \geq \begin{cases} 
2 \cdot 5^m & n \equiv 0 \bmod 4, \\
4 \cdot 5^m & n \equiv 1 \bmod 4, \\
6 \cdot 5^m & n \equiv 2 \bmod 4, \\
9 \cdot 5^m & n \equiv 3 \bmod 4,
\end{cases}$$
where $m=  \lfloor \frac{n}{4} \rfloor -2$. It is easy to verify that $\ns(C_n) \geq 2 \cdot 5^{n/4-2}$ in each of the four cases, which completes our proof.

\item[2.] For the inequality between $R_n$ and $S_n$, consider a tree $T$ of order $n$ for which the maximum $S_n$ is attained, i.e., $\ns(T) = S_n$. Let $v$ be a centroid vertex of $T$, which is a vertex for which the sum of the distances to all other vertices is minimized. It is well known that none of the centroid branches (the connected components that remain when the centroid is removed) can contain more than $n/2$ vertices, since one could then decrease the sum of distances by moving one step towards the largest branch (see Zelinka's paper \cite{zelinka}). Let $T_1,T_2,\ldots,T_k$ be the centroid branches, so that $|T_1| + |T_2| + \cdots + |T_k| = n-1$. The total number of subtrees that do not contain the centroid $v$ is clearly at most
$$\sum_{i=1}^k 2^{|T_i|}.$$
Some of these subtrees may of course be isomorphic, but we are only interested in an upper bound. Note that this sum increases if we transfer vertices from any of the branches to a branch with the same or greater number of vertices. Therefore, it reaches its maximum when there are only two branches, each containing either $(n-1)/2$ vertices (if $n$ is odd) or $n/2$ and $n/2-1$ vertices respectively (if $n$ is even). It follows that at most $2^{n/2} + 2^{n/2-1} = 3 \cdot 2^{n/2-1}$ of the subtrees of $T$ do not contain the centroid $v$. The number of distinct nonisomorphic subtrees containing $v$ is clearly at most $R_n$ by definition, so this completes the proof.

\item[3.] For the proof of the third statement, we use induction on $n$ to prove a minimally stronger inequality, which makes the inductive argument simpler. Specifically, we claim that for a rooted tree $T$ of order $n$,
\begin{equation}\label{eq:aux_ineq}
\nr(T) \leq 5^{n/4} - 1,
\end{equation}
unless $T$ is one of ten exceptional trees (denoted $E_1,\ldots,E_{10}$ for future reference) that are shown in Figure~\ref{fig:except}. Note, however, that all these trees still satisfy inequality~\eqref{eq:aux_ineq} without the final $-1$, which is what we actually want to obtain.

\begin{figure}[htbp]
\begin{center}
\begin{tikzpicture}
        \node[fill=black,rectangle,inner sep=2pt]  at (0,0) {};

	\node at (0,-1) {$\nr(E_1) = 1$};

        \node[fill=black,rectangle,inner sep=2pt]  at (2.5,0) {};
        \node[fill=black,circle,inner sep=1pt]  at (2.5,1) {};
	\draw (2.5,0)--(2.5,1);

	\node at (2.5,-1) {$\nr(E_2) = 2$};

        \node[fill=black,rectangle,inner sep=2pt]  at (5,0) {};
        \node[fill=black,circle,inner sep=1pt]  at (5,1) {};
        \node[fill=black,circle,inner sep=1pt]  at (5,2) {};
	\draw (5,0)--(5,2);

	\node at (5,-1) {$\nr(E_3) = 3$};

        \node[fill=black,rectangle,inner sep=2pt]  at (7.5,0) {};
        \node[fill=black,circle,inner sep=1pt]  at (7,1) {};
        \node[fill=black,circle,inner sep=1pt]  at (8,1) {};
	\draw (7,1)--(7.5,0)--(8,1);

	\node at (7.5,-1) {$\nr(E_4) = 3$};

        \node[fill=black,rectangle,inner sep=2pt]  at (10,0) {};
        \node[fill=black,circle,inner sep=1pt]  at (9.5,1) {};
        \node[fill=black,circle,inner sep=1pt]  at (10.5,1) {};
        \node[fill=black,circle,inner sep=1pt]  at (10.5,2) {};
	\draw (9.5,1)--(10,0)--(10.5,1)--(10.5,2);

	\node at (10,-1) {$\nr(E_5) = 5$};

        \node[fill=black,rectangle,inner sep=2pt]  at (0,-5) {};
        \node[fill=black,circle,inner sep=1pt]  at (-.5,-4) {};
        \node[fill=black,circle,inner sep=1pt]  at (.5,-4) {};
        \node[fill=black,circle,inner sep=1pt]  at (-1,-3) {};
        \node[fill=black,circle,inner sep=1pt]  at (0,-3) {};
	\draw (.5,-4)--(0,-5)--(-1,-3);
	\draw (-.5,-4)--(0,-3);

	\node at (0,-6) {$\nr(E_6) = 7$};

        \node[fill=black,rectangle,inner sep=2pt]  at (2.5,-5) {};
        \node[fill=black,circle,inner sep=1pt]  at (2,-4) {};
        \node[fill=black,circle,inner sep=1pt]  at (3,-4) {};
        \node[fill=black,circle,inner sep=1pt]  at (2,-3) {};
        \node[fill=black,circle,inner sep=1pt]  at (2,-2) {};
	\draw (3,-4)--(2.5,-5)--(2,-4)--(2,-2);

	\node at (2.5,-6) {$\nr(E_7) = 7$};

        \node[fill=black,rectangle,inner sep=2pt]  at (5,-5) {};
        \node[fill=black,circle,inner sep=1pt]  at (4.5,-4) {};
        \node[fill=black,circle,inner sep=1pt]  at (5,-4) {};
        \node[fill=black,circle,inner sep=1pt]  at (5.5,-4) {};
        \node[fill=black,circle,inner sep=1pt]  at (4.5,-3) {};
	\draw (5,-4)--(5,-5)--(4.5,-4)--(4.5,-3);
	\draw (5,-5)--(5.5,-4);

	\node at (5,-6) {$\nr(E_8) = 7$};

        \node[fill=black,rectangle,inner sep=2pt]  at (7.5,-5) {};
        \node[fill=black,circle,inner sep=1pt]  at (7,-4) {};
        \node[fill=black,circle,inner sep=1pt]  at (8,-4) {};
        \node[fill=black,circle,inner sep=1pt]  at (6.5,-3) {};
        \node[fill=black,circle,inner sep=1pt]  at (7.5,-3) {};
        \node[fill=black,circle,inner sep=1pt]  at (6.5,-2) {};
	\draw (8,-4)--(7.5,-5)--(6.5,-3)--(6.5,-2);
	\draw (7,-4)--(7.5,-3);

	\node at (7.5,-6) {$\nr(E_9) = 11$};

        \node[fill=black,rectangle,inner sep=2pt]  at (10,-5) {};
        \node[fill=black,circle,inner sep=1pt]  at (9.5,-4) {};
        \node[fill=black,circle,inner sep=1pt]  at (10,-4) {};
        \node[fill=black,circle,inner sep=1pt]  at (10.5,-4) {};
        \node[fill=black,circle,inner sep=1pt]  at (9,-3) {};
        \node[fill=black,circle,inner sep=1pt]  at (10,-3) {};
        \node[fill=black,circle,inner sep=1pt]  at (9,-2) {};
	\draw (10.5,-4)--(10,-5)--(9,-3)--(9,-2);
	\draw (9.5,-4)--(10,-3);
	\draw (10,-4)--(10,-5);

	\node at (10,-6) {$\nr(E_{10}) = 16$};

\end{tikzpicture}
\end{center}
\caption{The ten exceptional rooted trees.}\label{fig:except}
\end{figure}
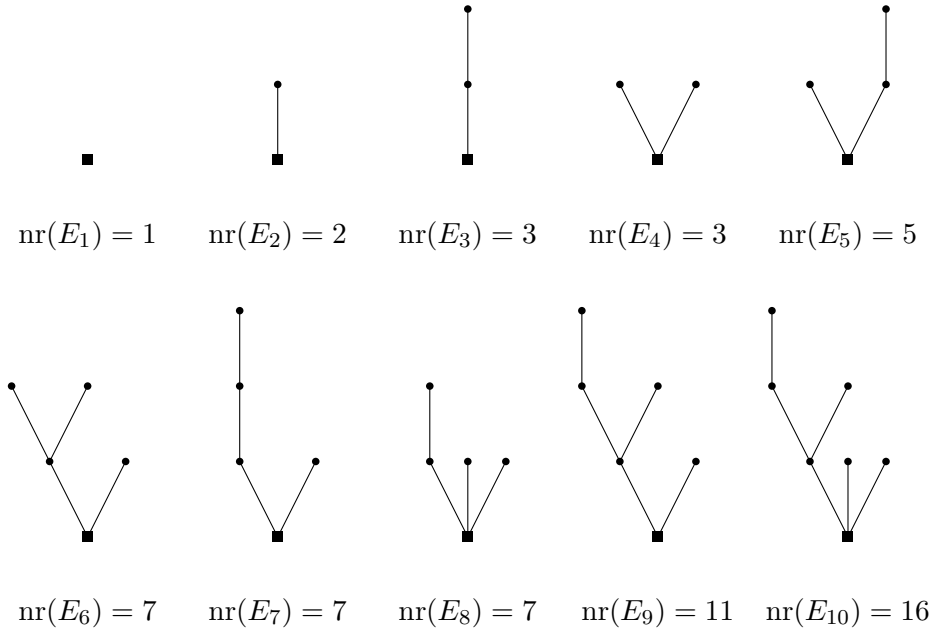

The statement can be verified directly for $n \leq 7$, so for the induction step, we consider a tree $T$ of order $n \geq 8$, and we denote its root by $r$ and its root branches by $T_1,T_2,\ldots,T_d$. Now consider the following cases:
\begin{itemize}
\item[\bf{Case 1:}] At least one of the branches (without loss of generality $T_1$) is not an exceptional tree. In this case, we can regard $T$ as the union of the rooted trees $R_1 = \{r\} \cup T_1$ and $R_2 = T \setminus T_1$, both rooted at $r$. If the tree $R_2$ is nontrivial (in other words, if $T$ has more than one branch), then we can apply \eqref{eq:upper_bound1} and the induction hypothesis to obtain
\begin{align*}
\nr(T) &\leq \nr(R_1) \nr(R_2) -1 = (\nr(T_1) + 1) \nr(R_2) -1  \\
&\leq (5^{|T_1|/4} - 1 + 1) \cdot 5^{|R_2|/4} - 1 \\
&= 5^{(|T_1|+|R_2|)/4} - 1 = 5^{n/4} - 1,
\end{align*}
which proves the desired inequality. If $T_1$ is the only branch, then we obtain from the induction hypothesis that
$$\nr(T) = 1 + \nr(T_1) \leq 1 + 5^{(n-1)/4} - 1 = 5^{(n-1)/4} \leq 5^{n/4} - 1.$$
\item[\bf{Case 2:}] We are left with the case that all branches are on the list of exceptional trees. Suppose that some 
set of branches (without loss of generality $T_1,T_2,\ldots,T_k$), together with the root of $T$, form a rooted tree $R_1$ such that $\nr(R_1) \leq 5^{(|R_1|-1)/4}$. In this case, we can apply the same argument as in the previous case: if $R_1$ is already all of $T$, we are done immediately; otherwise, set $R_2 = T \setminus R_1 \cup \{r\}$ and apply~\eqref{eq:upper_bound1} in combination with the induction hypothesis as before. This means that we are done if any of the following cases applies:
\begin{itemize}
\item At least four branches are single vertices (copies of $E_1$): in this case, $|R_1| = 5$ and $\nr(R_1) = 5$.
\item At least three branches have order $2$ (copies of $E_2$): in this case, $|R_1| = 7$ and $\nr(R_1) = 10$.
\item At least two branches are identical copies of one of the exceptional trees $E_j$, where $j \in \{3,4,\ldots,10\}$: in this case, $|R_1| = 2|E_j|+1$ and $\nr(R_1) = (\nr(E_j)+1)(\nr(E_j)+2)/2$, since each subtree of $R_1$ that contains the root is obtained from an unordered pair of root-containing subtrees of $E_j$, or a single such subtree, or consists of the root only. For each $j$, the desired inequality is easily verified.
\item At least two of the branches belong to the set $\{E_6,E_7,E_8,E_9,E_{10}\}$ of ``large'' exceptional branches: in each of these cases, one verifies that the tree $R_1$ formed by the root and these two branches satisfies $\nr(R_1) \leq 5^{(|R_1|-1)/4}$.
\end{itemize}
This leaves us with $4 \cdot 3 \cdot 2^3 \cdot 6 = 576$ remaining cases (determined by how often each exceptional tree occurs as a branch: up to three copies of $E_1$, up to two copies of $E_2$, either one or no copy for each of $E_3,E_4,E_5$, and potentially one of $E_6,E_7,\ldots,E_{10}$), and these can be checked directly by a computer. 
\end{itemize}

\end{enumerate}
\end{proof}

Our main result follows immediately:

\begin{theo}
We have both $S_n = \Theta(5^{n/4})$ and $R_n = \Theta(5^{n/4})$.
\end{theo}

\begin{proof}
Simply combine the inequalities of Proposition~\ref{prop:main} (and note that $2^{1/2} < 5^{1/4}$).
\end{proof}

\begin{rem}
Note the special role of the rooted tree $E_5$ in the construction of the trees $C_n$: these trees, which gave us the lower bound, mostly consist of copies of $E_5$, attached to a long path. The reason why this construction is essentially optimal is the fact that $5^{1/4}$ is the maximum of $\nr(T)^{1/|T|}$ taken over all rooted trees $T$, and this maximum is only attained by $E_5$.
\end{rem}

\begin{rem}
Both the upper and lower bound on $S_n$ and $R_n$ are probably not even asymptotically sharp. The following question is therefore natural:
\begin{center}
Does the limit $\lim_{n \to \infty} 5^{-n/4} S_n$ exist, and if so, what is its value?
\end{center}
It is conceivable that the limit does not exist in this form, but that it does exist if $n$ is restricted to a specific residue class modulo $4$ (compare the construction of the tree $C_n$, which depends on the residue class of $n$ modulo $4$).
\end{rem}


\begin{thebibliography}{10}

\bibitem{BubeckLinial} S. Bubeck and N. Linial,  On the local profiles of trees, J. Graph Theory, to appear (2015). DOI: 10.1002/jgt.21865

\bibitem{universal}
F.R.K. Chung, R.L. Graham and D.  Coppersmith,   ``On trees containing all small trees", The Theory of Applications of Graphs,  G. Chartrand, (Editor), John Wiley and Sons,   (1981), pp. 265--272


\bibitem{jamison1983average} R.E. Jamison,  On the average number of nodes in a subtree of a tree, J. Combin. Theory Ser. B 35(3)(1983), 207--223.

\bibitem{jamison1984monotonicity} R.E. Jamison,  Monotonicity of the mean order of subtrees, J. Combin. Theory Ser. B 37(1)(1984), 70--78. 




\bibitem{subtrees2005}  L.A.  Sz\'ekely and Hua Wang,   On subtrees of trees,
 Adv. Appl. Math. 34(2005), 138--155. 
 
  \bibitem{largest2006} L.A.  Sz\'ekely and Hua Wang,
 Binary trees with the largest number of
subtrees,   Discrete Appl. Math. 155(3)(2006), 374--385. 

\bibitem{zelinka} B. Zelinka,  Medians and peripherans of trees,  Arch. Math. (Brno) 4(1968), 87--95.  

\end{thebibliography}
\end{document}